\documentclass[journal,transmag]{IEEEtran}

\usepackage{cite}
\usepackage{hyperref}
\usepackage{amsmath}
\usepackage{amsthm}
\usepackage{amsfonts,mathrsfs}
\usepackage{amssymb}
\usepackage{graphicx}
\usepackage{enumitem}
\usepackage{array}

\newtheorem{theorem}{Theorem}
\newtheorem{lemma}[theorem]{Lemma}

\newtheorem{corollary}[theorem]{Corollary}
\newtheorem{proposition}[theorem]{Proposition}

\theoremstyle{definition}
\newtheorem{definition}[theorem]{Definition}

\newtheoremstyle{TheoremNum}
{\topsep}{\topsep}              
{\itshape}                      
{}                              
{\bfseries}                     
{.}                             
{ }                             
{\thmname{#1}\thmnote{ \bfseries #3}}
\newtheorem{remark}{Remark}

\usepackage{url}

\newcommand{\cD}{\mathcal D}
\newcommand{\F}{\mathbb F}

\newcommand{\Z}{\mathbb Z}
\newcommand{\K}{\mathbb K}
\newcommand{\cG}{\mathcal G}

\newcommand{\cC}{\mathcal C}
\newcommand{\bu}{\mathbf u}
\newcommand{\cH}{\mathcal H}

\newcommand{\bbS}{\mathbb S}

\newcommand{\Aut}{\mathrm{Aut}}

\newcommand{\GL}{\mathrm{GL}}

\newcommand{\PG}{\mathrm{PG}}
\newcommand{\Tr}{ \ensuremath{ \mathrm{Tr}}}
\newcommand{\lp}[2]{\mathscr{L}_{(#1,#2)}}
\newcommand{\RN}[1]{%
	\textup{\uppercase\expandafter{\romannumeral#1}}%
}

\begin{document}
%
\title{A new family of MRD codes in $\F_q^{2n\times2n}$ with\\ right and middle nuclei $\F_{q^n}$}


\author{\IEEEauthorblockN{Rocco Trombetti\IEEEauthorrefmark{1} and
Yue Zhou\IEEEauthorrefmark{2}}
\IEEEauthorblockA{\IEEEauthorrefmark{1}Dipartimento di Mathematica e Applicazioni ``R. Caccioppoli", Universit\`{a} degli \\Studi di Napoli ``Federico \RN{2}", I-80126 Napoli, Italy}
\IEEEauthorblockA{\IEEEauthorrefmark{2}College of Liberal Arts and Sciences, National University of Defense Technology, 410073 Changsha, China}
\thanks{This work is supported by the Research Project of MIUR (Italian Office for University and Research) ``Strutture Geometriche, Combinatoria e loro Applicazioni" 2012. Yue Zhou is supported by the National Natural Science Foundation of China (No.\ 11771451, 11531002). 
	
Corresponding author: Y. Zhou (email: yue.zhou.ovgu@gmail.com) 
	
Copyright (c) 2017 IEEE. Personal use of this material is permitted.}}

\markboth{Journal of \LaTeX\ Class Files,~Vol.~XX, No.~X, June~20XX}%
{Shell \MakeLowercase{\textit{et al.}}: Bare Demo of IEEEtran.cls for IEEE Transactions on Magnetics Journals}
%


\IEEEtitleabstractindextext{%
\begin{abstract}
	In this paper, we present a new family of maximum rank-distance (MRD for short) codes in $\F_q^{2n\times 2n}$ of minimum distance $2\leq d\leq 2n$. In particular, when $d=2n$, we can show that the corresponding semifield is exactly a Hughes-Kleinfeld semifield. The middle and right nuclei of these MRD codes are both equal to $\F_{q^n}$. We also prove that the MRD codes of minimum distance $2<d<2n$ in this family are inequivalent to all known ones. The equivalence between any two members of this new family is also determined.
\end{abstract}

}

\maketitle

\IEEEdisplaynontitleabstractindextext

%
\IEEEpeerreviewmaketitle

\section{Introduction}\label{sec:intro}
Let $\K$ denote a field. The set of all $m\times n$ matrices over $\K$ forms a $\K$-vector space, which we denote by $\K^{m\times n}$. For $A,B\in \K^{m\times n}$, we define 
$$d(A,B)=\mathrm{rank}(A-B),$$
which is often called the \emph{rank metric} or the \emph{rank distance} on $\K^{m\times n}$.

A subset $\cC\subseteq \K^{m\times n}$ with respect to the rank metric is called a \emph{rank-metric code} or a \emph{rank-distance code}. If $\cC$ contains at least two elements, the \emph{minimum distance} of $\cC$ is given by
\[d(\cC)=\min_{A,B\in \cC, A\neq B} \{d(A,B)\}.\]
When $\cC$ is a $\K$-linear subspace of $\K^{m\times n}$, we say that $\cC$ is a $\K$-linear code and its dimension $\dim_{K}(\cC)$ is defined to be the dimension of $\cC$ as a subspace over $\K$. 

Let $\F_q$ denote the finite field of $q$ elements. For any $\cC\subseteq \F_q^{m\times n}$ with $d(\cC)=d$, it is well-known that 
$$\#\cC\le q^{\max\{m,n\}(\min\{m,n\}-d+1)},$$
which is the Singleton-like bound for the rank metric; see \cite{delsarte_bilinear_1978}. When equality holds, we call $\cC$ a \emph{maximum rank-distance} (MRD for short) code. More properties of MRD codes can be found in \cite{delsarte_bilinear_1978},~\cite{gabidulin_MRD_1985},~\cite{gadouleau_properties_2006},~\cite{morrison_equivalence_2014} and~\cite{ravagnani_rank-metric_2016}.

Rank-metric codes, in particular, MRD codes have been studied since the 1970s and have seen much interest in recent years due to a wide range of applications including storage systems~\cite{roth_1991_Maximum}, cryptosystems~\cite{gabidulin_public-key_1995}, spacetime codes~\cite{lusina_maximum_2003} and random linear network coding~\cite{koetter_coding_2008}. 

In finite geometry, there are several interesting structures including quasifields, semifields, and splitting dimensional dual hyperovals can be equivalently described as special types of rank-metric codes; see~\cite{dempwolff_dimensional_2014},~\cite{dempwolff_orthogonal_2015},~\cite{johnson_handbook_2007}, \cite{taniguchi_unified_2014} and the references therein. In particular, a finite quasifield corresponds to an MRD code in $\F_q^{n\times n}$ of minimum distance~$n$ and a finite semifield corresponds to such an MRD code that is a subgroup of $\F_q^{n\times n}$ (see~\cite{de_la_cruz_algebraic_2016} for the precise relationship). Many essentially different families of finite quasifields and semifields are known \cite{lavrauw_semifields_2011}, which yield many inequivalent MRD codes in $\K^{n\times n}$ of minimum distance~$n$. In contrast, it appears to be much more difficult to obtain inequivalent MRD codes in $\F_q^{n\times n}$ of minimum distance strictly less than $n$. For the relationship between MRD codes and other geometric objects such as linear sets and Segre varieties, we refer to \cite{lunardon_mrd-codes_2017}.

Besides quasifields, there are only a few known constructions of MRD codes in $\F_q^{n\times n}$. The first construction of MRD codes was given by Delsarte~\cite{delsarte_bilinear_1978}. This construction was later rediscovered by Gabidulin~\cite{gabidulin_MRD_1985} and generalized by Kshevetskiy and Gabidulin~\cite{kshevetskiy_new_2005}. Today this family is usually called the \emph{generalized Gabidulin codes}, sometimes it is also simply called the \emph{Gabidulin codes} (see Section~\ref{sec:pre}, for a precise definition). It is easy to show that a Gabidulin code is always $\F_{q^n}$-linear. Recently, another $\F_q$-linear family was found by Sheekey~\cite{sheekey_new_2016} and we often call them (generalized) twisted Gabidulin codes. This family has been further generalized into additive MRD codes by Otal and \"Ozbudak~\cite{otal_additive_2016}, who also constructed a family of non-additive MRD codes \cite{otal_non-additive_2018}. Given any $2\leq d \leq n$, all these constructions can provide us MRD codes of minimum distance $d$.

For MRD codes in $\F_q^{n\times n}$ of minimum distance $d=n-1$, there are a few more constructions. First, there is a nonlinear family constructed by Cossidente, Marino and Pavese~\cite{cossidente_non-linear_2016} and later generalized by Durante and Siciliano~\cite{durante_nonlinear_MRD_2017}. Besides this family, there are other constructions associated with maximum scattered linear sets over $\PG(1,q^6)$ and $\PG(1,q^8)$ presented recently in \cite{csajbok_newMRD_2017} and~\cite{csajbok_maximum_arxiv}. For more results concerning maximum scattered linear sets and associated MRD codes, see~\cite{bartoli_scattered_arxiv}, \cite{csajbok_classes_2018}, \cite{csajbok_equivalence_2016} and \cite{csajbok_maximum_4_2018}.

For MRD codes in $\F_q^{m\times n}$ with $m<n$, there are many different approaches to construct them. A canonical way to get them is puncturing (or projecting) MRD codes in $\F_{q}^{m'\times n}$ with $n\geq m'>m$. In \cite{horlemann-trautmann_new-criteria_2017}, a new criterion for the punctured Gabidulin codes is presented, and for small $m$ and $n$, several constructions of inequivalent MRD codes are obtained. In \cite{neri_genericity_2018}, it is presented a generic construction of MRD codes by using algebraic geometry approaches, under the condition that $n$ is large enough compared with $d$ and $m$. In \cite{csajbok_maximum_2017}, an approach to derive MRD codes in $\F_q^{m\times n}$ from linear sets is investigated. In \cite{donati_generalization_2017}, a nonlinear construction is presented. Recently, Schmidt and the second author~\cite{schmidt_number_MRD_2017} showed that even in Gabidulin codes there are a huge subset of inequivalent MRD codes.

In this paper, we present a new family of MRD codes in $\F_{q}^{2n\times 2n}$ of any minimum distance $d$ between $2$ and $2n$. In particular, when $d=2n$, we can show that the corresponding semifield is exactly the Hughes-Kleinfeld semifield~\cite{hughes_seminuclear_1960} found in 1960. Through the investigation of their middle and right nuclei, we can prove that the MRD codes in this new family are inequivalent to all known constructions.

The rest of this paper is organized as follows. In Section \ref{sec:pre}, we introduce semifields, describe rank-metric codes in $\F_{q}^{n\times n}$ via linearized polynomials and introduce the equivalence between rank-metric codes as well as their dual codes and adjoint codes. In Section \ref{sec:construction}, we present our new family of MRD codes and determine their middle and right nuclei. Based on these results, we show that they are inequivalent to all the known MRD codes except for one special case which is later excluded in Section \ref{sec:equivalence}. Another result in Section \ref{sec:equivalence} is the complete answer to the equivalence problem between different members of this new family.

\section{Preliminaries}\label{sec:pre}
Roughly speaking, a \emph{semifield} $\bbS$ is an algebraic structure satisfying all the axioms of a skewfield except (possibly) the associativity of its multiplication. A finite field is a trivial example of a semifield. Furthermore, if $\bbS$ does not necessarily have a multiplicative identity, then it is called a \emph{presemifield}.  For a presemifield $\bbS$, $(\bbS,+)$ is necessarily abelian \cite{knuth_finite_1965}.

The first family of non-trivial semifields was constructed by Dickson \cite{dickson_commutative_1906} more than a century ago. In \cite{knuth_finite_1965}, Knuth showed that the additive group of a finite semifield $\bbS$ is an elementary abelian group, and the additive order of the nonzero elements in $\bbS$ is called the \emph{characteristic} of $\bbS$. Hence, any finite semifield can be represented by $(\mathbb{F}_q, +, *)$, where $q$ is a power of a prime $p$. Here $(\mathbb{F}_q, +)$ is the additive group of the finite field $\mathbb{F}_q$ and $x*y$ can be written as $x*y=\sum_{i,j}a_{ij} x^{p^i}y^{p^j}$, which forms a map from $\mathbb{F}_q\times \mathbb{F}_q$ to $\mathbb{F}_q$. We refer to \cite{lavrauw_semifields_2011} for a recent and comprehensive survey on finite semifields.

Geometrically speaking, there is a well-known correspondence, via coordinatisation, between (pre)semifields and projective planes of Lenz-Barlotti type V.1, see \cite{dembowski_finite_1997,hughes_projective_1973}. The most important equivalence relation defined on (pre)semifields is the \emph{isotopism}.  Given two (pre)semifields $\bbS_1=(\mathbb{F}_p^n, +, *)$ and $\bbS_2=(\mathbb{F}_p^n, +, \star)$. If there exist three bijective linear mappings $L, M, N:\mathbb{F}_{p}^n\rightarrow \mathbb{F}_p^n$ such that
$$M(x)\star N(y)=L(x*y)$$
for any $x,y\in\mathbb{F}_p^n$, then $\bbS_1$ and $\bbS_2$ are called \emph{isotopic}, and the triple $(M,N,L)$ is called  an \emph{isotopism} between $\bbS_1$ and $\bbS_2$. In \cite{albert_finite_1960}, Albert showed that two (pre)semifields coordinatize isomorphic planes if and only if they are isotopic. Every presemifield can be normalized into a semifield under an appropriate isotopism; see~\cite{bierbrauer_projective_2016} and~\cite{lavrauw_semifields_2011}.

Given a semifield $\bbS$ with multiplication $*$, we define its left, middle and right nucleus by
\begin{align*}
N_l(\bbS) &:= \{a\in \bbS: a*(x* y) = (a* x) * y \text{ for all }x,y\in \bbS \},\\
N_m(\bbS) &:= \{a\in \bbS: x*(a* y) = (x* a) * y \text{ for all }x,y\in \bbS \},\\
N_r(\bbS) &:= \{a\in \bbS: x*(y* a) = (x* y) * a \text{ for all }x,y\in \bbS \}.
\end{align*}
It is not difficult to prove that the semifield $\bbS$ can be viewed as a left vector space over its left nucleus. In particular, when $\bbS$ is finite, we can further show $N_l(\bbS)$ is actually a finite field $\F_q$. Let us assume that $\bbS$ is of size $q^n$. For every $b\in \bbS$, the map $x\mapsto x*b$ defines an $n\times n$ matrix $M_b$ over 
its left nucleus $N_l(\bbS)$. Furthermore, all such matrices together form a rank metric code $\{M_b: b\in \bbS \}$ which is actually an MRD code, because the difference  between any two distinct members $M_b$ and $M_d$ in it equals $M_b-M_d=M_{b-d}$ which is always nonsingular. This MRD code is usually called the semifield spread set associated with $\bbS$; see \cite{lavrauw_semifields_2011}.

Next, let us turn to rank-metric codes. 
As we are working with rank-metric codes in $\F_q^{n\times n}$ rather than $\F_{q}^{m\times n}$ with $m<n$ in this paper, it is more convenient to describe such a rank-metric code  using the language of $q$-polynomials or linearized polynomials over $\F_{q^n}$ which are the polynomials in the set
\[\lp{n}{q}[X]=\left\{\sum c_i X^{q^i}: c_i\in \F_{q^n} \right\}.\]
In fact, there is a bijection between $\F_q^{n\times n}$ and $\lp{n}{q}[X]/(X^{q^n}-X)$; for more results about linearized polynomials, we refer to \cite{lidl_finite_1997}.

As we mentioned in the introduction part, the most well-known family of MRD codes is called (generalized) Gabidulin codes. They can be described by the following set $\cG_{k,s}$ of linearized polynomials
\[\small{\{a_0 x + a_1 x^{q^{s}} + \dots a_{k-1} x^{q^{s(k-1)}}: a_0,a_1,\dots, a_{k-1}\in \F_{q^n} \},}\]
where $s$ is relatively prime to $n$. It is obvious that there are $q^{kn}$ polynomials in $\cG_{k,s}$ and each polynomial in it has at most $q^{k-1}$ roots which means its minimum distance $d=n-k+1$. Hence its size meets the Singleton-like bound.

For $x\in \F_{q^{m}}$, let $N_{q^{m}/q}(x)$ denote the norm from $\F_{q^m}$ to $\F_q$, i.e.\ $N_{q^{m}/q}(x)=x^{1+q+\cdots q^{m-1}}$. The following result follows from \cite[Theorem 10]{gow_galois_linear_2009}.
\begin{lemma}\label{lm:gow}
	Let $s$ and $m$ be two relatively prime positive integers. Suppose that $f=f_0X+f_1X^{q^s} +\cdots+ f_kX^{q^{sk}}\in \F_{q^m}[X]$ is a linearized polynomial with $f_k\neq 0$. If $f$ has $q^k$ roots, then $N_{q^{sm}/q^s}(f_0)=(-1)^{km}N_{q^{sm}/q^s}(f_k)$.
\end{lemma}

In \cite{sheekey_new_2016}, Sheekey applied Lemma \ref{lm:gow} and found a new family of MRD codes $\cH_{k,s}(\eta, h)$ which equals
\[\small{\{a_0 x + a_1 x^{q^s} + \dots +a_{k-1} x^{q^{s(k-1)}} + \eta a_0^{q^h} x^{q^{sk}}: a_i\in \F_{q^m} \},}\]
where $\eta$ satisfies $N_{q^{sm}/q^s}(\eta)\neq (-1)^{km}$, i.e.\ $N_{q^{m}/q}(\eta)\neq (-1)^{km}$.
Such an MRD code is usually called a \emph{(generalized) twisted Gabidulin code}. 
It is clear that if we allow $\eta$ equal $0$, $\cG_{k,s}$ can be viewed as a subfamily of the twisted Gabidulin codes. Replacing the field automorphism $a_0\mapsto a_0^{q^h}$ in the coefficient of the last term of the elements in $\cH_{k,s}(\eta, h)$ by an automorphism in $\Aut(\F_{q^n})\setminus \Aut(\F_{q^n}/\F_q)$, Otal and \"Ozbudak~\cite{otal_additive_2016} generalized this family into an additive one.

There are several slightly different definitions of equivalence of rank-metric codes. In this paper, we use the following notion of equivalence.
\begin{definition}
	\label{def:equivalence}
	Two rank-metric codes $\cC_1$ and $\cC_2$ in $\K^{m\times n}$ are \emph{equivalent} if there exist $A\in\GL_m(\K)$, $B\in \GL_n(\K)$, $C\in\K^{m\times n}$ and $\rho\in\Aut(\K)$ such that 
	\begin{equation}\label{eq:def_equiv}
	\cC_2=\{AM^{\rho}B+C:M \in\cC_1\}.
	\end{equation}
	For $m=n$, if $\cC_2$ is equivalent to $\cC_1$ or $\cC_1^T := \{M^T: M\in \cC_1\}$
	where $(\,.\,)^T$ means transposition, then we say $\cC_1$ and $\cC_2$ are \emph{isometrically equivalent}. An equivalence map from a rank-metric code $\cC$ to itself is also called an \emph{automorphism} of $\cC$.
\end{definition}

When $\cC_1$ and $\cC_2$ are both additive and equivalent, it is not difficult to show that we can choose $C=0$ in \eqref{eq:def_equiv}. In particular, when $\cC_1$ and $\cC_2$ are semifield spread sets, they are equivalent if and only if the associated semifields are isotopic \cite[Theorem 7]{lavrauw_semifields_2011}.

Back to the descriptions in linearized polynomials, given two rank-metric codes $\cC_1$ and $\cC_2$ which consist of linearized polynomials, they are equivalent if there exist $\varphi_1$, $\varphi_2\in \lp{n}{q}[X]$ permuting $\F_{q^n}$, $\psi\in \lp{n}{q}[X]$ and $\rho\in \Aut(\F_q)$ such that
\[ \varphi_1\circ f^\rho \circ \varphi_2 + \psi\in \cC_2 \text{ for all }f\in \cC_1,\]
where $\circ$ stands for the composition of maps and $f^\rho= \sum a_i^\rho X^{q^i}$ for $f=\sum a_i X^{q^i}$.

In general, it is a difficult problem to tell whether two given rank-metric codes are equivalent or not. There are several invariants which may help us to distinguish them. 
Given a $\K$-linear rank-metric code $\cC\subseteq \K^{m\times n}$, its middle nucleus is defined as
\[N_m(\cC) =\{M\in\K^{m\times n} : MC\in \cC \text{ for all }C\in \cC  \},\]
and its right nucleus is defined as
\[N_r(\cC) =\{M\in\K^{m\times n} : CM\in \cC \text{ for all }C\in \cC  \}.\]
These two concepts were introduced in \cite{lunardon_kernels_2017} and they can be viewed as a natural generalization of the middle and right nucleus of semifields.
In \cite{liebhold_automorphism_2016}, they are called the left idealizer and the right idealizer of $\cC$, respectively. In general, we can also define the left nucleus of $\cC$.  However, for MRD codes over $\K$ containing singular matrices, it is always $\K$ which means it is not a useful invariant; see \cite{lunardon_kernels_2017}.

For a rank-metric code $\cC$ given by a set of linearized polynomials, its middle nucleus and right nucleus can also be written as sets of linearized polynomials. Precisely the middle nucleus of $\cC$ is
\[\mathcal N_m(\cC)= \{ \varphi \in \lp{n}{q}: f\circ \varphi\in \cC \text{ for all }f\in \cC \}.\]
It is defined by $f\circ \varphi$ rather than $\varphi\circ f$ because we always consider a row vector $\bu$ multiplying a matrix $C$ which is a member of a rank-metric code. This means that $M\in N_m(\cC)$ only if $\bu MC=\bu C'$ for some $C'\in \cC$.

Similarly, the right nucleus of $\cC$ is
\[\mathcal N_r(\cC)= \{ \varphi \in \lp{n}{q}: \varphi \circ f\in \cC \text{ for all }f\in \cC \}.\] 
They played an important role in \cite{schmidt_number_MRD_2017} proving a lower bound on the numbers of inequivalent Gabidulin codes in $\F_q^{m\times n}$. The middle and right nuclei of generalized twisted Gabidulin codes together with a complete answer to the equivalence between members in this family can be found in \cite{lunardon_generalized_2018}.

We define a symmetric bilinear form on the set $\F_q^{m\times n}$ by
\[\langle M,N\rangle:= \Tr(MN^T),\]
where $N^T$ is the transpose of $N$. The \emph{Delsarte dual code} of an $\F_q$-linear code $\cC$ is 
\[\cC^\perp :=\{M\in \F_q^{m\times n}:\langle M,N \rangle=0 \text{ for all } N\in \cC  \}.\]

One important result proved by Delsarte \cite{delsarte_bilinear_1978} is that the Delsarte dual code of a linear MRD code is still MRD. As we are considering MRD codes using linearized polynomials, the Delsarte dual can also be interpreted in the following way~\cite{sheekey_new_2016}.
We define the bilinear form $b$ on $q$-polynomials by
\[b\left( f,g \right)=\Tr_{q^n/q}\left(\sum_{i=0}^{n-1}a_ib_i\right),\]
where $f(x)=\sum_{i=0}^{n-1}a_ix^{q^i}$ and $g(x)=\sum_{i=0}^{n-1}b_ix^{q^i}\in \F_{q^n}[x]$. The \emph{Delsarte dual code} $\cC^\perp$ of a set of $q$-polynomials $\cC$ is
\[\cC^\perp=\{f: b(f,g)=0 \text{ for all } g\in \cC \} .\]
It is well-known and also not difficult to show directly that two linear rank-metric codes are equivalent if and only if their duals are equivalent.

Let $\cC$ be an MRD codes in $\K^{m\times n}$. It is obvious that $\{M^T: M\in \cC\}$ is also an MRD codes, because the ranks of $M^T$ and $M$ are the same. When $\K=\F_q$ and $m=n$, we can also interpret the transpose of matrices into an operation on $q$-polynomials.

The \emph{adjoint} of a $q$-polynomial $f=\sum_{i=0}^{n-1}a_i x^{q^i}$ is given by 
$$\hat{f}:=\sum_{i=0}^{n-1}a_{i}^{q^{n-i}} x^{q^{n-i}}.$$ 
If $\cC$ is a rank-metric codes consisting of $q$-polynomials, then the \emph{adjoint code} of $\cC$ is $\widehat{\cC}:=\{\hat{f}: f\in\cC\}$. In fact, the adjoint of $f$ is equivalent to the transpose of the matrix derived from $f$. This result can be found in \cite{kantor_commutative_2003}.

Regarding the adjoint and Delsarte dual operation, we have
\begin{equation}\label{eq:operation_1}
\mathcal{N}_m(\widehat{\cC}) = \widehat{\mathcal{N}_r(\cC)}=\mathcal{N}_r(\cC^\perp),
\end{equation}
and
\begin{equation}\label{eq:operation_2}
\mathcal{N}_m(\cC^\perp) = \widehat{\mathcal{N}_m(\cC)}=\mathcal{N}_r(\widehat{\cC}),
\end{equation}
which are proved in \cite[Proposition 4.2]{lunardon_kernels_2017}.

\section{A class of MRD codes}\label{sec:construction}
In the rest of this paper, we write $N(x)$ instead of $N_{q^{2n}/q}(x)$ for short. By applying Lemma \ref{lm:gow}, we can get another family of MRD codes.

\begin{theorem}\label{th:construction1}
	Let $s$ and $n$ be two integers satisfying $\gcd(s,2n)=1$.	For $\gamma \in \F_{q^{2n}}$ satisfying that $N(\gamma)$ is a non-square in $\F_q$, we define $\cD_{k,s}(\gamma)$ as the set
	\begin{equation}\label{eq:D_gamma}
	\small{
	\left\{a X + \sum_{i=1}^{k-1}c_i X^{q^{is}} + \gamma bX^{q^{ks}} : c_1, \cdots, c_{k-1}\in \F_{q^{2n}}, a,b\in \F_{q^n} \right\}.}
	\end{equation}
	Then $\cD_{k,s}(\gamma)$ is an MRD code.
\end{theorem}
\begin{proof}
	It is clear that $\#\cD_{k,s}(\gamma)=q^{2nk}$. We need to show that for each polynomial $f\in \cD_{k,s}(\gamma)$, it has at most $q^{k-1}$ roots which means its minimum distance is $d=2n-k+1$. Hence $\cD_{k,s}(\gamma)$ is an MRD code.
	
	By way of contradiction, let us assume that $f=a X + \sum_{i=1}^{k-1}c_i X^{q^{is}} + \gamma bX^{q^{k}}$ has $q^k$ roots which implies that $a$ and $b$ are both nonzero. By Lemma \ref{lm:gow}, $N_{q^{2sn}/q^{s}}(a)=(-1)^{2nk}N_{q^{2sn}/q^{s}}(\gamma b)$. Hence $N_{q^{2sn}/q^{s}}(a/b)=N_{q^{2sn}/q^{s}}(\gamma) = N(\gamma)$. As $a,b\in \F_{q^n}$, 
	\[N_{q^{2sn}/q^{s}}\left(\frac{a}{b}\right)=\left(\frac{a}{b}\right)^{2(1+q^s+\cdots +q^{(n-1)s})}=N_{q^n/q}\left(\frac{a}{b}\right)^2\]
	which is a square in $\F_q$. However, this contradicts the assumption on $N(\gamma)$.
\end{proof}

Let us first look at the Delsarte dual code of $\cD_{k,s}(\gamma)$. It is straightforward to compute that $\cD_{k,s}(\gamma)^\perp$ equals
\[\left\{-\gamma bX + aX^{q^{ks}} + \sum_{i=k+1}^{2n-1} c_i X^{q^{is}}:  a,b\in \F_{q^n}, c_i\in \F_{q^{2n}}  \right\}.\]
Replacing $X$ by $X^{q^{2n-ks}}$ in every term and module $X^{q^{2n}}-X$, we get $\cD_{2n-k,s}(-\gamma)$.
\begin{proposition}\label{prop:dual}
	The Delsarte dual code of $\cD_{k,s}(\gamma)$ is equivalent to $\cD_{2n-k,s}(-\gamma)$.
\end{proposition}
It can also be readily verified the following result of the adjoint code of $\cD_{k,s}(\gamma)$.
\begin{proposition}\label{prop:adjoint}
	The adjoint code of $\cD_{k,s}(\gamma)$ is equivalent to $\cD_{k,s}(1/\gamma)$.
\end{proposition}

By Theorem \ref{th:construction1}, it is clear that $aX+\gamma bX^{q^{s}}$ defines a semifield multiplication. As $\gamma\notin \F_{q^n}$ (otherwise $N(\gamma)$ must be a square in $\F_q$), for every $x\in \F_{q^{2n}}$, we can write it as $x=c+d\gamma$ for some $c,d\in \F_{q^{n}}$. 

Assume that $\gamma^{q^s+1}=u+v\gamma$ for certain $u,v\in \F_{q^n}$. 
By expanding $ax+\gamma bx^{q^{s}}$, we have
\[ax+\gamma bx^{q^{s}}= (ac+bd^{q^s}u) + (ad+bc^{q^s} + bd^{q^s}v)\gamma. \]
We view them as vectors in $\F_{q^n}^2$ and define a semifield multiplication
\begin{equation}\label{eq:semi_multi}
(c,d)*(a,b)=(ac+bd^{q^s}u , ad+bc^{q^s} + bd^{q^s}v),
\end{equation}
for $a,b,c,d\in \F_{q^n}$. By comparing with \cite[Theorem 9.7]{hughes_projective_1973}, we see that \eqref{eq:semi_multi} is exactly the multiplication of a Hughes-Kleinfeld semifield \cite{hughes_seminuclear_1960}, which is also the multiplication of a Knuth semifield of type \RN{2} \cite{knuth_finite_1965}.

By \cite[Lemma 9.8]{hughes_projective_1973}, $\F_{q^n}$ is the right and middle nucleus of $\mathbb{H}$. In \cite{hughes_collineation_1960}, a necessary and sufficient condition for $x+y\gamma \in N_l(\mathbb{H})$ are derived.
\begin{proposition}\label{prop:3nuclei}
	Let $*$ be the multiplication defined by \eqref{eq:semi_multi} and $\mathbb{H}$ denote the associated semifield $(\F_{q^n}^2, +, *)$.
	\begin{enumerate}[label=(\alph*)]
		\item $N_r(\mathbb{H}) = \F_{q^n}$.
		\item $N_m(\mathbb{H}) = \F_{q^n}$.
		\item For $x,y\in \F_{q^n}$, $x+y\gamma \in N_l(\mathbb{H})$ if and only if
		\[ \left\{
		\begin{array}{l}
		x^{q^{2s}}+y^{q^{2s}}v^{q^s}=x+y^{q^s}v,\\ 
		yu + x^{q^s}v +y^{q^s}v^2 = y^{q^{2s}} u^{q^s} + x^{q^{2s}} v + y^{q^{2s}} v^{q^s+1}.
		\end{array} 
		\right.  \]
	\end{enumerate}
\end{proposition}

Next, let us investigate the middle and right nucleus of the MRD codes $\cD_{k,s}(\gamma)$ defined in Theorem \ref{th:construction1}. They are very important invariants with respect to the equivalence of rank-metric codes. We will use them later to show that $\cD_{k,s}(\gamma)$ contains MRD codes which are not equivalent to any known one.
\begin{theorem}\label{th:N_rN_m_construction1}
	Let $k$ be an integer satisfying $1\le k\le  2n-1$. Then the right nucleus of $\cD_{k,s}(\gamma)$ is
	\begin{equation}\label{eq:N_r(D)}
	\mathcal N_r(\cD_{k,s}(\gamma))= \{aX : a\in \F_{q^n}\},
	\end{equation}
	and its middle nucleus is
	\begin{equation}\label{eq:N_m(D)}
	\mathcal N_m(\cD_{k,s}(\gamma))= \{aX : a\in \F_{q^n}\}.
	\end{equation}
\end{theorem}
\begin{proof}
	When $k=1$, $\cD_{k,s}(\gamma)$ is isotopic to a Hughes-Kleinfeld semifield $\mathbb H$, and the result can be then derived from Proposition \ref{prop:3nuclei}. By duality, we get the result for $k=2n-1$.
	
	In the rest part, we assume that $2\leq k \leq 2n-2$. Assume that $\varphi=\sum_{i=0}^{2n-1}d_iX^{q^{is}}$ is an element in $\mathcal N_r(\cD_{k,s}(\gamma))$. As $\varphi(c_1 X^{q^s})\in \cD_{k,s}(\gamma)$, we see that $d_i=0$ for $k<i<2n-1$. In fact, only $d_{2n-1}$, $d_0$ and $d_1$ can be nonzero. When $k=2$, this is obvious. When $k>2$, this statement can be directly verified by checking $\varphi(c_jX^{q^{js}})\in \cD_{k,s}(\gamma)$ for $j=2,\cdots, k-1$.
	
	Next we consider $\varphi(aX+ \gamma bX^{q^{ks}})$ which should also be in $\cD_{k,s}(\gamma)$. As $\varphi = d_{2n-1}X^{q^{(2n-1)s}}+d_0X +d_1X^{q^s}$, in the expansion of $\varphi(aX+ \gamma bX^{q^{ks}})$ the coefficient of $X^{q^{(2n-1)s}}$ is $d_{2n-1}a^{q^{(2n-1)s}}$ and the coefficient of $X^{q^{(k+1)s}}$ is $ d_1(\gamma b)^{q^s}$. Since $a$ and $b$ can take any value in $\F_{q^n}$, by checking the elements in $\cD_{k,s}(\gamma)$ we see that $d_{2n-1}$ and $d_1$ must be $0$. Thus $\varphi=d_0X$. From
	\[\varphi(aX+ \gamma bX^{q^{ks}})=d_0aX+ \gamma d_0bX^{q^{ks}}\in \cD_{k,s}(\gamma) \]
	we derive $d_0\in \F_{q^n}$.
	
	By $\widehat{\cD_{k,s}}(\gamma)=\cD_{k,s}(1/\gamma)$ in Proposition \ref{prop:adjoint} and \eqref{eq:operation_1}, we get
	\begin{align*}
	\mathcal N_m(\cD_{k,s}(\gamma)) &=\mathcal{N}_r\left(\cD_{k,s}\left(\frac{1}{\gamma}\right)^\perp\right)\\
	&=\mathcal{N}_r\left(\cD_{2n-k,s}\left(-\frac{1}{\gamma^{q^{2n-ks}}}\right)\right)	
	\end{align*}
	which equals $\{aX : a\in \F_{q^n}\}$.
\end{proof}
\begin{corollary}\label{coro:inequivalence}
	In $\F_{q}^{2n\times 2n}$, the MRD code $\cD_{k,s}(\gamma)$ is not equivalent to any generalized Gabidulin code. When $k\neq n$ or $h\neq n$, $\cD_{k,s}(\gamma)$ is also not equivalent to any generalized twisted Gabidulin code $\cH_{k,t}(\eta, h)$ with $\eta\neq 0$.
\end{corollary}
\begin{proof}
	When $k=1$, a (generalized) Gabidulin code $\{aX : a\in\F_{q^{2n}} \}$ is derived from the multiplication of the finite field $\F_{q^{2n}}$, which is never isotopic to a Hughes-Kleinfeld semifield. Hence the corresponding MRD codes are not equivalent.
	
	Moreover, $\cH_{1,s}(\eta,h)$ associates a generalized twisted field, which is a presemifield. If $s\not \equiv h \pmod{2n}$, it is isotopic to a semifield whose middle nucleus is of size $q^{\gcd(2n,s-h)}$ and right nucleus is of size $q^{\gcd(2n, h)}$; otherwise $s\equiv h \pmod{2n}$ and this presemifield is isotopic to the finite field $\F_{q^{2n}}$; see~\cite{albert_generalized_1961} and~\cite{biliotti_collineation_1999}. Thus $\cH_{1,s}(\eta,h)$ cannot be equivalent to $\cD_{1,s}(\gamma)$.
	
	When $k=2n-1$, we can simply consider the equivalences between their Delsarte dual codes which have been already determined above.
	
	In the rest, we only have to investigate the equivalence problem for $1<k<2n-1$.
	
	According to \cite[Corollary 5.9 (a)]{lunardon_kernels_2017}, the middle (or right) nucleus of a generalized Gabidulin codes over $\F_{q^{2n}}$ is always  $\F_{q^{2n}}$. Hence it is different from the middle nucleus of $\cD_{k,s}(\gamma)$ by Theorem \ref{th:N_rN_m_construction1}, which means that they are not equivalent.
	
	According to \cite[Corollary 5.9 (b)]{lunardon_kernels_2017},
	$\mathcal N_m(\cH_{k,t}(\eta, h))$ in $\F_q^{2n\times 2n}$ is of size $q^{\gcd(2n, sk-h)}$ and $\mathcal N_r(\cH_{k,t}(\eta, h))$ is of size $q^{\gcd(2n, h)}$. Hence, $\cD_{k,s}(\gamma)$ and $\cH_{k,t}(\eta, h)$ are equivalent only if $\gcd(2n,h)=h$ and $\gcd(2n, sk-h)=n$ which means $h=n=k$.
\end{proof}

For $k=2$, as the middle nucleus of $\cD_{2,s}(\gamma)$ is not of size $q^{2n}$, it is also not equivalent to those MRD codes associated with maximum scattered linear sets constructed in \cite{csajbok_newMRD_2017} and \cite{csajbok_maximum_arxiv}.

\section{Equivalence}\label{sec:equivalence}
In Section \ref{sec:construction}, we have shown that most members of the MRD codes $\cD_{k,s}(\gamma)$ are new with respect to the equivalence of rank-metric codes. In the last part of this section, we will completely solve the last open case whether $\cD_{n,s}(\gamma)$ is equivalent to $\cH_{n,t}(\theta, n)$ or not.

First we investigate the equivalence between different members of this family. 
If we want to further determine the isometric equivalence between $\cD_{k,s}(\gamma)$ and $\cD_{k,t}(\theta)$, the answer follows directly from our result about the equivalence map from $\cD_{k,s}(\gamma)$ to $\cD_{k,t}(\theta)$ or its adjoint $\cD_{k,t}(1/\theta)$.

By using our knowledge of the middle and right nucleus of $\cD_{k,s}(\gamma)$, we can prove the following results.
\begin{lemma}\label{lm:equivalence_map}
	Let $n,s,t\in \Z^+$ satisfying $\gcd(2n,s)=\gcd(2n,t)=1$. Let $\gamma$ and $\theta$ be in $\F_{q^{2n}}$ satisfying that $N(\gamma)$ and $N(\theta)$ are both non-square in $\F_q$. Let $(\varphi_1,\varphi_2, \rho)$ be an equivalence map between $\cD_{k,s}(\gamma)$ and $\cD_{k,t}(\theta)$ for $1<k<2n-1$. If $k\neq n$ or $n\geq3$, then $\varphi_1$ and $\varphi_2$ are both monomials.
\end{lemma}
\begin{proof}
	By Proposition \ref{prop:dual}, $\cD_{2n-k,s}(-\gamma)$ is equivalent to  the Delsarte dual code of $\cD_{k,s}(\gamma)$. As two MRD codes are equivalent if and only if their Delsarte duals are equivalent, we only have to prove the statement for $k\le n$. 
	
	According to the definition of equivalence, $\varphi_1 \circ f^\rho \circ \varphi_2\in \cD_{k,t}(\theta)$ for every $f\in \cD_{k,s}(\gamma)$. As $f^\rho\in \cD_{k,s}(\gamma^\rho)$, $\varphi_1$ must be in the normalizer of $\mathcal N_r(\cD_{k,s}(\gamma^\rho))$ in $\GL(2n,q)$. By Theorem \ref{th:N_rN_m_construction1}, the right nucleus $\mathcal N_r(\cD_{k,s}(\gamma^\rho))=\{aX : a\in \F_{q^n}\}$. It follows that 
	$$\varphi_1=cX^{q^l} + dX^{q^{l+n}}$$ 
	for a certain $l\in \{0,\cdots, 2n-1\}$ and $c,d\in \F_{q^{2n}}$. This result is well-known and can be verified directly as follows. Assume that $\varphi_1= \sum_{i=0}^{2n-1}a_iX^{q^i}$. Then for each $b\in \F_{q^n}$, there always exists some $b'\in \F_{q^n}$ such that
	\[ \varphi_1\circ bX= \sum_{i=0}^{2n-1}a_ib^{q^i}X^{q^i}=b'X\circ \varphi_1=\sum_{i=0}^{2n-1}b'a_iX^{q^i}.\]
	This implies that if $a_i\neq 0$ then $b^{q^i}=b'$, which means that at most two coefficients $a_l$ and $a_{l+n}$ are nonzero for a certain $l$.
	
	By the same argument, we can also show that $$\varphi_2=gX^{q^j} + hX^{q^{j+n}}$$ 
	for some $j\in \{0,\cdots, 2n-1\}$ and $g,h\in \F_{q^{2n}}$.
	
	Now let us look at the image of $c_iX^{q^{is}}$ under the equivalence map $(\varphi_1, \varphi_2, \mathrm{id})$. By calculation,
	\begin{align}
	\nonumber  &\varphi_1 \circ c_iX^{q^{is}} \circ \varphi_2\\
	\nonumber =&cc_i^{q^l}(gX^{q^j} + hX^{q^{j+n}})^{q^{is+l}} + dc_i^{q^{l+n}}(gX^{q^j} + hX^{q^{j+n}})^{q^{is+l+n}}\\
	\label{eq:binomials}=& (cg^{q^{is+l}} c_i^{q^l} + dh^{q^{is+l+n}}c_i^{q^{l+n}})X^{q^{j+is+l}} \\
	&\nonumber + (ch^{q^{is+l}} c_i^{q^l} + dg^{q^{is+l+n}}c_i^{q^{l+n}}) X^{q^{j+is+l+n}}. 
	\end{align}
	
	When $k<n$, as $\varphi_1 \circ c_iX^{q^{is}} \circ \varphi_2\in \cD_{k,t}(\theta)$, one of the coefficients of $X^{q^{j+is+l}}$ and $X^{q^{j+is+l+n}}=X^{q^{j+is+l+tn}}$ must be $0$ for all $c_i\in \F_{q^{2n}}$. Together with the condition that $\varphi_1$ and $\varphi_2$ are permutation polynomials, we show that $c=h=0$ or $d=g=0$ or $c=g=0$ or $d=h=0$, which means $\varphi_1$ and $\varphi_2$ are both monomials.
	
	When $k=n$, as $\varphi_1 \circ c_iX^{q^{is}} \circ \varphi_2\in \cD_{k,t}(\theta)$, the coefficients of $X^{q^{j+is+l}}$ and $X^{q^{j+is+l+n}}=X^{q^{j+is+l+tn}}$ are both nonzero only if exact one of $j+is+l$ and $j+is+l+tn$ equals $0$. If $n\geq3$, $i$ can be taken for at least two different values from $\{1,2,\cdots,k-1\}$. Hence we can choose the value of $i$ such that $j+is+l\neq 0,n$. As in the case $k<n$, we see that $\varphi_1$ and $\varphi_2$ must be both monomials.
\end{proof}

The following lemma can be proved by using exactly the same argument and we omit its proof.
\begin{lemma}\label{lm:equivalence_map_D_H}
	Let $n,s,t\in \Z^+$ satisfying $\gcd(2n,s)=\gcd(2n,t)=1$. Let $\gamma$ and $\eta$ be in $\F^*_{q^{2n}}$ satisfying that $N(\gamma)$ is a non-square in $\F_q$ and $N(\eta)\neq 1$. Let $(\varphi_1,\varphi_2, \rho)$ be an equivalence map between $\cD_{n,s}(\gamma)$ and $\cH_{n,t}(\eta, n)$. If $n\geq3$, then $\varphi_1$ and $\varphi_2$ are both monomials.
\end{lemma}

Now we can determine the equivalence between $\cD_{k,s}(\gamma)$ and $\cD_{k,t}(\theta)$.
\begin{theorem}\label{th:equivalence}
	Let $n,s,t\in \Z^+$ satisfying $\gcd(2n,s)=\gcd(2n,t)=1$. Let $\gamma$ and $\eta$ be in $\F_{q^{2n}}$ satisfying that $N(\gamma)$ and $N(\theta)$ are both non-square in $\F_q$. Let $k$ be an integer satisfying $1<k<2n-1$.
	
	When $k\neq n$ or $n\geq 3$, the MRD code $\cD_{k,s}(\gamma)$ is equivalent to  $\cD_{k,t}(\theta)$ if and only if one of the following collections of conditions are satisfied.
	\begin{enumerate}[label=(\alph*)]
		\item $s\equiv t \pmod{2n}$, there exist $\sigma \in \Aut(\F_{q^{2n}})$ and $h\in \F_{q^{2n}}$ such that $\gamma^{\sigma} h^{q^{ks}-1}=\theta$.
		\item $s\equiv -t \pmod{2n}$, there exist $\sigma \in \Aut(\F_{q^{2n}})$ and $h\in \F_{q^{2n}}$ such that $\gamma^{\sigma} h^{q^{ks}-1}=1/\theta$.
	\end{enumerate}
\end{theorem}
\begin{proof}
	As in the proof of Lemma \ref{lm:equivalence_map}, we only have to handle the cases $k\leq n$.
	
	Assume that $(\varphi_1, \varphi_2, \rho)$ is an equivalence map between $\cD_{k,s}(\gamma)$ and $\cD_{k,t}(\theta)$.
	
	When $k\neq n$ or $n\geq 3$, by Lemma \ref{lm:equivalence_map}, we can assume that $\varphi_1=dX^{q^l}$ and $\varphi_2 = g X^{q^j}$ for some $d, g\in \F_{q^n}^*$ and $l,j\in \{0,\cdots, 2n-1\}$.
	
	For arbitrary $c_i\in \F_{q^{2n}}$ with $i\in \{1,\cdots, k-1 \}$, 
	\[  \varphi_1\circ c_iX^{q^{is}} \circ \varphi_2 = dc_i^{q^l} g^{q^{is+l}} X^{q^{is+l+j}}.\]
	It follows that $is+l+j\in \{t,2t, \cdots, (k-1)t\}$. As $\gcd(2n,s)=\gcd(2n,t)=1$, we can assume that $s\equiv rt \pmod{2n}$ which means
	\[ \{irt+l+j \pmod{2n}: i =1,\cdots, k-1 \}= \{t,2t, \cdots, (k-1)t\}. \]
	As $k\le n$, it is straightforward to see that either $r=1$ and $l+j\equiv 0 \pmod{2n}$, or $r=-1$ and $l+j\equiv kt \pmod{2n}$.
	
	When $r=1$, i.e.\ $s\equiv t\pmod{2n}$ and $j\equiv -l \pmod{2n}$,  for $a,b\in \F_{q^n}$, applying $(\varphi_1, \varphi_2, \rho)$ onto $aX+ \gamma bX^{q^{ks}}$, we obtain
	\begin{eqnarray*}
	\lefteqn{\varphi_1\circ \left(a^\rho X+ \gamma^\rho b^\rho X^{q^{ks}} \right)\circ \varphi_2 }\\
	&=& da^{\rho q^l}g^{q^l}X+ d\gamma^{\rho q^l}b^{\rho q^l}g^{q^{ks+l}}X^{q^{ks}},
	\end{eqnarray*}
	which belongs to $\cD_{k,t}(\theta)$ if and only if $dg^{q^l}\in \F_{q^n}$ and $d\gamma^{\rho q^l}g^{q^{ks+l}}\in \theta \F_{q^n}$. Let $\sigma$ denote the automorphism of $\F_{q^n}$ defined by $x\mapsto x^{\rho q^l}$. Let $h = g^{q^l}$. Then we see that there must be a solution of $h$ such that $\gamma^{\sigma} h^{q^{ks}-1}=\theta$.
	
	When $r=-1$, i.e.\ $s\equiv -t\pmod{2n}$ and $j\equiv kt-l \pmod{2n}$, for $a,b\in \F_{q^n}$, we apply $(\varphi_1, \varphi_2, \rho)$ onto $aX+ \gamma bX^{q^{ks}}$ and get
	\begin{eqnarray*}
	\lefteqn{\varphi_1\circ \left(a^\rho X+ \gamma^\rho b^\rho X^{q^{ks}} \right)\circ \varphi_2}\\
	&=& da^{\rho q^l}g^{q^l}X^{q^{kt}}+ d\gamma^{\rho q^l}b^{\rho q^l}g^{q^{ks+l}}X,
	\end{eqnarray*}
	which belongs to $\cD_{k,t}(\theta)$ if and only if $dg^{q^l}\in\theta \F_{q^n}$ and $d\gamma^{\rho q^l}g^{q^{ks+l}}\in \F_{q^n}$. Let $\sigma$ denote the automorphism of $\F_{q^n}$ defined by $x\mapsto x^{\rho q^l}$. Let $h = g^{q^l}$. Then we see that there must be a solution of $h$ such that $\gamma^{\sigma} h^{q^{ks}-1}=1/\theta$.
	
	Therefore we have proved the necessary condition in the statement for $k\neq n$ or $n\geq3$. For sufficiency, it is routine to do a  verification.
\end{proof}

There are 3 cases which are not covered by Theorem \ref{th:equivalence}: $k=1$, $k=2n-1$ and $k=n=2$. 
For $k=1$, $\cD_{1,s}(\gamma)$ defines a Hughes-Kleinfeld semifield whose autotopism group has been completely determined in \cite{hughes_collineation_1960}. It appears that using the same approach, the equivalence between $\cD_{1,s}(\gamma)$ and $\cD_{1,t}(\theta)$ can also be determined. Hence, in the rest of this section, we will skip the case $k=1$. Moreover, for $k=2n-1$, the MRD code is the Delsarte dual code of a Hughes-Kleinfeld semifield by Proposition \ref{prop:dual}. We will also skip this case, because the equivalence problem for this case can be completely converted into the equivalence problem for Hughes-Kleinfeld semifields.

Next we investigate the last case in which $k=n=2$. 
As $\gcd(2n,s)=\gcd(2n,t)=1$ and $n=2$, $t\equiv \pm s \pmod{2n}$. In fact, $t$ and  $s$ can only be $1$ or $-1$ modulo $2n$.
\begin{theorem}\label{th:equivalence_n=k}
	Let $s,t\in \Z^+$ satisfying $\gcd(4,s)=\gcd(4,t)=1$. Let $\gamma$ and $\theta$ be in $\F_{q^{4}}$ satisfying that $N_{q^{4}/q}(\gamma)$ and $N_{q^{4}/q}(\theta)$ are both non-square in $\F_q$. 
	
	The MRD code $\cD_{2,s}(\gamma)$ is equivalent to  $\cD_{2,t}(\theta)$ if and only if one of the following collections of conditions are satisfied.
	\begin{enumerate}[label=(\alph*)]
		\item $s\equiv t \pmod{4}$, there exists $\sigma \in \Aut(\F_{q^{4}})$ and $h\in \F_{q^{4}}$ such that $\gamma^{\sigma} h^{q^{2s}-1}=\theta$.
		\item $s\equiv t \pmod{4}$, there exist $c,d,g,h\in\F_{q^4}$, $\rho\in \Aut(\F_q)$ and $l\in \{0,1,2,3\}$ such that
		\[ \left\{
		\begin{array}{rcl}
		cg^{q^{s+l}}-d^{q^2}h^{q^{s+l}}&=&0,	\\ 
		ch^{q^{s+l}}\theta^{q^2}-d^{q^2}g^{q^{s+l}}\theta&=&0,    \\ 
		cg^{q^l}+ dh^{q^{l+2}}&=&0,\\ 
		ch^{q^{2s+l}}\gamma^{\rho q^l}+ dg^{q^{l}}\gamma^{\rho q^{l+2}}&=&0.
		\end{array} 
		\right.  \]
		\item $s\equiv -t \pmod{4}$, there exists $\sigma \in \Aut(\F_{q^{4}})$ and $h\in \F_{q^{4}}$ such that $\gamma^{\sigma} h^{q^{2s}-1}=1/\theta$.
		\item $s\equiv -t \pmod{4}$, there exist $c,d,g,h\in\F_{q^4}$, $\rho\in \Aut(\F_q)$ and $l\in \{0,1,2,3\}$ such that
		\[ \left\{
		\begin{array}{rcl}
		cg^{q^{s+l}}-d^{q^2}h^{q^{s+l}}&=&0,	\\ 
		ch^{q^{s+l}}\theta^{q^2}-d^{q^2}g^{q^{s+l}}\theta&=&0,    \\ 
		ch^{q^l}+ dg^{q^{l+2}}&=&0,\\ 
		cg^{q^{2s+l}}\gamma^{\rho q^l}+ dh^{q^{l}}\gamma^{\rho q^{l+2}}&=&0.
		\end{array} 
		\right.  \]
	\end{enumerate}
\end{theorem}
\begin{proof}
	In this proof, we will still write $n$ instead of $2$ in some equations even though we have assumed that $n=2$.
	
	If $cX^{q^s}\in \cD_{2,s}(\gamma)$ is always mapped to another monomial for all $c\in \F_{q^{2n}}$, then the same calculation in Theorem \ref{th:equivalence} shows the necessary and sufficient conditions (a) and (c).
	
	In the rest of the proof, we always assume that $cX^{q^s}\in \cD_{2,s}(\gamma)$ is mapped to a binomial for some $c$. 
	Taking $i=1$ in \eqref{eq:binomials}, we see that $j+s+l$ can be taken for exact two possible value:  $j+s+l\equiv 0 \pmod{2n}$ or $j+s+l\equiv n \pmod{2n}$. 
	
	Let us consider the case $s\equiv t \pmod{2n}$. 
	First we assume that $j+s+l\equiv 0 \pmod{2n}$.
	From $\varphi_1\circ c_1X^{q^{s}} \circ \varphi_2\in\cD_{2,t}(\theta)$ and \eqref{eq:binomials}, we derive that the coefficient of $X^{q^{j+s+l}}$ belongs to $\F_{q^n}$ and the coefficient of $X^{q^{j+s+l+n}}$ belongs to $\theta\F_{q^n}$, which means that
	\begin{eqnarray}
	\label{eq:k=n=2_qs_1}\lefteqn{cg^{q^{s+l}} c_1^{q^l} + dh^{q^{s+l+n}}c_1^{q^{l+n}}} \\
	\nonumber	& = & c^{q^n}g^{q^{s+l+n}} c_1^{q^{l+n}} + d^{q^n}h^{q^{s+l}}c_1^{q^{l}}
	\end{eqnarray}
	and
	\begin{eqnarray}
	\label{eq:k=n=2_qs_2}	\lefteqn{(ch^{q^{s+l}}c_1^{q^l} + d g^{q^{s+l+n}} c_1^{q^{l+n}})\theta^{q^n}}\\
	\nonumber & = &(c^{q^n}h^{q^{s+l+n}}c_1^{q^{l+n}} + d^{q^n} g^{q^{s+l}} c_1^{q^{l}})\theta
	\end{eqnarray}
	hold for every $c_1\in\F_{q^{2n}}$. If we view \eqref{eq:k=n=2_qs_1} as a polynomial of $c_1$, by comparing the coefficients of $c_1^{q^l}$ (or those of $c_1^{q^{l+n}}$) in it, we obtain
	\begin{equation}\label{eq:cgdh1}
	cg^{q^{s+l}}=d^{q^n}h^{q^{s+l}}.
	\end{equation}
	Similarly, from \eqref{eq:k=n=2_qs_2} we derive
	\begin{equation}\label{eq:chdg1}
	ch^{q^{s+l}}\theta^{q^n} = d^{q^n}g^{q^{s+l}}\theta.
	\end{equation}
	
	Furthermore, from $\varphi_1\circ aX\circ \varphi_2\in \cD_{2,t}(\theta)$ with $a\in \F_{q^n}$ we can derive more conditions. As the coefficient of $X^{q^{j+l}}=X^{q^{2n-s}}$ in it must be zero, by plugging $a=c_i$ and $i=0$ into \eqref{eq:binomials}, we get
	\begin{equation}\label{eq:cgdh2}
	cg^{q^l}+ dh^{q^{l+n}}=0.
	\end{equation}
	
	Analogously, by checking the coefficient of $X^{q^{j+2s+l+n}}=X^{q^{j+l}}=X^{q^{3s}}$ in $\varphi_1\circ \gamma^\rho bX^{q^{2s}}\circ \varphi_2\in \cD_{2,t}(\theta)$ with $b\in \F_{q^n}$, we obtain
	\begin{equation}\label{eq:chdg2}
	ch^{q^{2s+l}}\gamma^{\rho q^l}+ dg^{q^{l}}\gamma^{\rho q^{l+n}}=0.
	\end{equation}

	For $j+s+l\equiv n \pmod{2n}$, the proof is similar. By checking the coefficients of $X^{q^{j+s+l}}=X^{q^n}$ and $X^{q^{j+s+l+n}}=X$ in $\varphi_1\circ c_1X^{q^{s}} \circ \varphi_2\in\cD_{2,t}(\theta)$, we obtain
	\begin{align}
	\label{eq:-cgdh1} ch^{q^{s+l}}&=d^{q^n}g^{q^{s+l}},\\
	\label{eq:-chdg1} cg^{q^{s+l}}\theta^{q^n} &= d^{q^n}h^{q^{s+l}}\theta.	
	\end{align}
	
	Furthermore, as the coefficient of $X^{q^{j+l+n}}=X^{q^{3s}}$ in $\varphi_1\circ aX\circ \varphi_2\in \cD_{2,t}(\theta)$ must be $0$ for every $a\in \F_{q^n}$,
	\begin{equation}\label{eq:-cgdh2}
	ch^{q^l}+ dg^{q^{l+n}}=0.
	\end{equation}
	By checking the coefficient of $X^{q^{j+2s+l}}=X^{q^{3s}}$ in $\varphi_1\circ \gamma^\rho bX^{q^{2s}}\circ \varphi_2\in \cD_{2,t}(\theta)$ with $b\in \F_{q^n}$, we get
	\begin{equation}\label{eq:-chdg2}
	cg^{q^{2s+l}}\gamma^{\rho q^l}+ dh^{q^{l}}\gamma^{\rho q^{l+n}}=0.
	\end{equation}
	Hence, \eqref{eq:-cgdh1}, \eqref{eq:-cgdh2}, \eqref{eq:-chdg1} and \eqref{eq:-chdg2} can be simply obtained by switching of $g$ and $h$ in \eqref{eq:cgdh1}, \eqref{eq:cgdh2}, \eqref{eq:chdg1} and \eqref{eq:chdg2}, respectively. We finish the proof of the necessity part of (b).
	
	After a careful check of the previous calculations, we can see that if $c$, $d$, $g$, $h$, $\rho$ and $l$ satisfy \eqref{eq:cgdh1}, \eqref{eq:chdg1}, \eqref{eq:cgdh2} and \eqref{eq:chdg2} simultaneously, then the map $(\varphi_1, \varphi_2, \rho)$ is indeed an equivalence map between $\cD_{2,s}(\gamma)$ and $\cD_{2,t}(\theta)$. Therefore the condition (b) is also sufficient.
	
	For the case (d) in which $s\equiv -t \pmod{4}$, the proof is the same. For $j+s+l\equiv 0\pmod{2n}$, we can also get the same equations \eqref{eq:cgdh1} and \eqref{eq:chdg1}. However, now $g$ and $h$ are switched in \eqref{eq:cgdh2} and \eqref{eq:chdg2}. We omit the details of these calculations.
\end{proof}
\begin{remark}
	It is possible that the conditions (b) and (d) hold. For instance, let $q=3$, $s=1$ and $\gamma=\theta=\omega$ which is a root of $X^4+2X^3+2\in \F_{q}[X]$. Taking $l=0$, $c=1$, $d=\omega^{36}$, $g=\omega^2$, $h=\omega^{54}$ and $\rho=\mathrm{id}$, we get an equivalence map from $\cD_{2,1}(\gamma)$ to itself.
\end{remark}
\begin{remark}
	By Theorem \ref{th:equivalence} (a) and Theorem \ref{th:equivalence_n=k} (a) (b), the automorphism group of an MRD code $\cD_{k,s}(\gamma)$ can also be determined.
\end{remark}

Recall that in Corollary \ref{coro:inequivalence}, the equivalence between $\cD_{n,s}(\gamma)$ and $\cH_{n,t}(\eta, n)$ is the unique open case. Finally we will solve this problem by using the same approach which was used in the proofs of Theorems \ref{th:equivalence} and \ref{th:equivalence_n=k}.
\begin{theorem}\label{th:n=k_equivalence_D_H}
	Let $n,s,t\in \Z^+$ satisfying $\gcd(2n,s)=\gcd(2n,t)=1$. Let $\gamma$ and $\eta$ be in $\F^*_{q^{2n}}$ satisfying that $N(\gamma)$ is a non-square in $\F_q$ and $N(\eta)\neq 1$. Then $\cD_{k,s}(\gamma)$ and $\cH_{k,t}(\eta, h)$ are not equivalent for all $k$ and $h$.
\end{theorem}
\begin{proof}
	By Corollary \ref{coro:inequivalence}, we only have to face the case $k=n=h$. Assume that $(\varphi_1, \varphi_2,\rho)$ defines an equivalence map from $\cD_{n,s}(\gamma)$ to $\cH_{n,t}(\eta, n)$. As we are going to show that such a map never exist, without loss of generality, we assume that $\rho=\mathrm{id}$; otherwise we consider the equivalence map from $\cD_{n,s}(\gamma^\rho)$ to $\cH_{n,t}(\eta, n)$.
	
	We separate our proof into two parts depending on the value of $n$.
	
	(a) When $n\geq 3$, the proof is quite similar to that for Theorem \ref{th:equivalence}. By Lemma \ref{lm:equivalence_map_D_H}, we can assume that $\varphi_1=dX^{q^l}$ and $\varphi_2 = g X^{q^j}$ for some $d, g\in \F_{q^n}^*$ and $l,j\in \{0,\cdots, 2n-1\}$.
	
	For arbitrary $c_i\in \F_{q^{2n}}$ with $i\in \{1,\cdots, n-1 \}$, 
	\[  \varphi_1\circ c_iX^{q^{is}} \circ \varphi_2 = dc_i^{q^l} g^{q^{is+l}} X^{q^{is+l+j}},\]
	which should belong to $\cH_{n,t}(\eta, n)$.	It follows that $is+l+j\in \{t,2t, \cdots, (n-1)t\}$. As $\gcd(2n,s)=\gcd(2n,t)=1$, we can assume that $s=rt$ whence
	\[ \{irt+l+j : i =1,2,\cdots, n-1 \}= \{t,2t, \cdots, (n-1)t\}. \]
	It is straightforward to see that $r=1$ and $l+j\equiv 0 \pmod{2n}$, or $r=-1$ and $l+j\equiv nt \pmod{2n}$.
	
	No matter $r=1$ or $-1$, applying $(\varphi_1, \varphi_2, \mathrm{id})$ to $aX$, we see that one of the coefficients of $X$ and $X^{q^{kt}}$ is zero and the other one is a function of $a$. This contradicts the assumption $\varphi_1\circ aX\circ \varphi_2\in \cH_{n,t}(\eta, n)$ for every $a\in \F_{q^n}$.
	
	(b) When $n=2$, it is clear that $s$ and $t$ are congruent to $\pm 1$ modulo $2n$. In fact, it is sufficient to consider the case $t=s$, because $\cH_{n,-s}(\eta, n)$ is equivalent to $\cH_{n,s}(1/\eta^{q^{2s}},n)$.
	
	As the middle and right nuclei of $\cD_{n,s}(\gamma)$ to $\cH_{n,t}(\eta, n)$ are $\F_{q^n}$, we can assume that $\varphi_1=cX^{q^l} + dX^{q^{l+n}}$ and $\varphi_2=gX^{q^j} + hX^{q^{j+n}}$. Our first goal is to show that $\varphi_1$ and $\varphi_2$ must be monomials.
	
	Assume, by way of contradiction, that $c,d,g,h$ are all nonzero. Plugging $i=1$ and $c_i=w$ into \eqref{eq:binomials}, we get
	\begin{align}
	\label{eq:binomials-final} \varphi_1\circ wX^{q^s} \circ \varphi_2 = &(cg^{q^{s+l}} w^{q^l} + dh^{q^{s+l+n}}w^{q^{l+n}})X^{q^{j+s+l}}\\ 
	\nonumber		& + (ch^{q^{s+l}} w^{q^l} + dg^{q^{s+l+n}}w^{q^{l+n}}) X^{q^{j+s+l+n}},
	\end{align}
	which should belong to $\cH_{n,s}(\eta, n)$ for all $w\in \F_{q^{2n}}$.
	As in the proof of Theorem \ref{th:equivalence_n=k}, we see that $j+s+l$ can only take two possible value: $j+s+l\equiv 0 \pmod{2n}$ or $j+s+l\equiv n \pmod{2n}$.
	
	If $j+s+l\equiv 0 \pmod{2n}$, from 
	$\varphi_1\circ wX^{q^s} \circ \varphi_2\in \cH_{n,t}(\eta, n)$, we derive
	\[\eta(cg^{q^{s+l}} w^{q^l} + dh^{q^{s+l+n}}w^{q^{l+n}})^{q^n} = ch^{q^{s+l}} w^{q^l} + dg^{q^{s+l+n}}w^{q^{l+n}}\]
	for every $w\in \F_{q^{2n}}$, which means
	\[\left\{
	\begin{array}{l}
	\eta c^{q^{n}}g^{q^{3s+l}} = dg^{q^{3s+l}},\\
	\eta d^{q^n}h^{q^{s+l}}= ch^{q^{s+l}}.
	\end{array} 
	\right. \]
	As we have assumed that $g$ and $h$ are both nonzero, the two above equations implies that $\eta c^{q^n}=d$ and $\eta d^{q^n}=c$. Hence $\eta^{q^n+1}=1$, which implies that $N(\eta)= \eta \eta^q\eta^{q^2}\eta^{q^3}= (\eta \eta^{q^2})(\eta \eta^{q^2})^q=1$ contradicting the assumption that $N(\eta)\neq 1$.
	
	For $j+s+l\equiv n\pmod{2n}$, the proof is analogous and we omit it.
	
	Therefore we have proved that $\varphi_1=dX^{q^l}$ and $\varphi_2 = g X^{q^j}$ for some $d, g\in \F_{q^{2n}}^*$ and $l,j\in \{0,\cdots, 3\}$. As the case $n\ge 3$ proved in part (a), it is routine to expand $\varphi_1 \circ aX \circ \varphi_2$ and to check that it cannot belong to $\cH_{n,t}(\eta, n)$. Hence there is no equivalence map from $\cD_{n,s}(\gamma)$ to $\cH_{n,t}(\eta, n)$.
\end{proof}


\ifCLASSOPTIONcaptionsoff
  \newpage
\fi



%

%


\begin{IEEEbiographynophoto}{Rocco Trombetti}
was born in Caserta (Italy) in 1975. He received the Degree in Mathematics in 1997 from the University of Campania “Luigi Vanvitelli”, and the Ph.D in Mathematics in 2004 from the University of Naples ``Federico II", where he is currently a Professor. His research interests are in combinatorics, with particular regard to finite geometry. He obtained results, in collaboration also with Italian and foreign researchers, on the following topics: spreads and ovoids of polar spaces, semifields, non-associative algebras and associated geometric structures, MRD-codes.
\end{IEEEbiographynophoto}

\begin{IEEEbiographynophoto}{Yue Zhou}
was born in Taiyuan, China in 1984. He received the B.S.\ degree and M.S.\ degree in mathematics from the National University of Defense Technology in Changsha, China, in 2006 and 2009, and the PhD degree in mathematics from the Otto-von-Guericke University Magdeburg, Germany in 2013. Currently he is a lecturer at the National University of Defense Technology in Changsha, China. His research interests include finite geometries, combinatorial design theory, finite fields, codes and sequences.
\end{IEEEbiographynophoto}





\end{document}